\documentclass[10pt,reqno]{amsart}
\usepackage{amssymb,amsthm,amsfonts,amsmath,amscd,amssymb,epsfig}
\usepackage[all]{xy}

\usepackage{enumerate,array,hyperref, graphicx,color,amsthm,
bbm,mathrsfs,a4wide}

\theoremstyle{plain}
\newtheorem{theorem}{Theorem}[section]
\newtheorem{thm}[theorem]{Theorem}
\newtheorem{corollary}[theorem]{Corollary}
\newtheorem{cor}[theorem]{Corollary}
\newtheorem{proposition}[theorem]{Proposition}

\newtheorem{lemma}[theorem]{Lemma}

\theoremstyle{remark}

\newtheorem*{remark}{Remark}

\theoremstyle{definition}
\newtheorem*{definition}{Definition}
\newcommand{\id}{{{\mathchoice {\rm 1\mskip-4mu l} {\rm 1\mskip-4mu l}
{\rm 1\mskip-4.5mu l} {\rm 1\mskip-5mu l}}}}

\newcommand{\N}{{\mathbb{N}}}
\newcommand{\Z}{{\mathbb{Z}}}
\newcommand{\R}{{\mathbb{R}}}
\newcommand{\C}{{\mathbb{C}}}
\newcommand{\Q}{{\mathbb{Q}}}

\renewcommand{\P}{{\mathbb{P}}}
\newcommand{\comment}[1]{}

\DeclareMathOperator{\lcm}{lcm}

\begin{document}

\title[Displaceability and the mean Euler characteristic]
{Displaceability and the mean Euler characteristic}
\author{Urs Frauenfelder}
\author{Felix Schlenk}
\author{Otto van Koert}
\address{
    Urs Frauenfelder\\
    Department of Mathematics and Research Institute of Mathematics\\
    Seoul National University}
\email{frauenf@snu.ac.kr}
\address{
Felix Schlenk\\
Institut de Math\'ematiques,
Universit\'e de Neuch\^atel,
Rue \'Emile Argand~11,
2000 Neuch\^atel, Switzerland}
\email{schlenk@unine.ch}

\address{Otto van Koert\\
Department of Mathematics and Research Institute of Mathematics\\
Seoul National University}
\email{okoert@snu.ac.kr}


\begin{abstract}
In this note we show that the mean Euler characteristic of equivariant symplectic homology is an effective obstruction against the existence of displaceable exact contact embeddings.
As an application we show that certain Brieskorn manifolds do not admit displaceable exact contact embeddings.
\end{abstract}

\maketitle

\section[Introduction]{Introduction}

A contact manifold $(\Sigma,\xi)$ is said to admit an \emph{exact contact embedding}
if there exists an embedding $\iota \colon \Sigma \to V$ into an exact symplectic manifold
$(V,\lambda)$

and a contact form $\alpha$ for~$(\Sigma,\xi)$ such that
$\alpha - \iota^*\lambda$ is exact, and such that $\iota (\Sigma) \subset V$ is bounding.
In this paper we suppose, in addition, that any target manifold $(V,\lambda)$ is convex,
i.e., there exists an exhaustion $V= \bigcup_{k} V_k$ of~$V$ by compact sets $V_k \subset V_{k+1}$
with smooth boundary such that $\lambda |_{\partial V_k}$
is a contact form,
and that the first Chern class of $(V, \lambda)$ vanishes on $\pi_2(V)$.
An exact contact embedding  is called \emph{displaceable} if $\iota (\Sigma)$
can be displaced from itself by a Hamiltonian isotopy of~$V$.
We refer to~\cite{CF} for more details on exact contact embeddings,
and for examples and obstructions to such embeddings.

The \emph{mean Euler characteristic} of a simply-connected contact manifold
was introduced in the third author's thesis~\cite{vK}
in terms of contact homology,
and was studied further in~\cite{E,GK}.
Here, we shall consider the mean Euler characteristic of equivariant symplectic homology,
which can be thought of as the mean Euler characteristic of a filling.
For the definition see Section~\ref{s:mean}.
Under additional assumptions, these notions coincide,
see Corollary~\ref{cor:mec_invariant} and the subsequent remark.

We say that a simply-connected cooriented contact manifold $(\Sigma,\alpha)$ is
\emph{index-positive} if the mean index $\Delta(\gamma)$ of every periodic Reeb orbit~$\gamma$ is positive.
Similarly, we say that $(\Sigma,\alpha)$ is
\emph{index-negative} if the mean index $\Delta(\gamma)$ of every periodic Reeb orbit $\gamma$ is negative.
Finally, we say that $(\Sigma,\alpha)$ is \emph{index-definite}
if it is index-positive or index-negative.
Recall that
the mean index $\Delta$ is related to the Conley--Zehnder index $\mu_{CZ}$ as follows:
For any non-degenerate Reeb orbit $\gamma$ in a contact manifold $(\Sigma^{2n-1},\alpha)$, its $N$-fold cover $\gamma^N$ satisfies
\begin{equation}
\label{eq:iteration_and_mean_index}
\mu_{CZ}(\gamma^N) \,=\, N \Delta(\gamma) + e(N),
\end{equation}
where $e(N)$ is an error term bounded by~$n-1$, see~\cite{SZ}[Lemma 3.4].

In this note we prove the following theorem.

\medskip
\textbf{Theorem~A. }
{\it
Assume that $(\Sigma,\xi)$ is a $(2n-1)$-dimensional simply-connected contact manifold
which admits a displaceable exact contact embedding.
Suppose furthermore that $(\Sigma,\alpha)$ is index-definite for some $\alpha$ defining $\xi$.
Then the following holds.

\begin{itemize}
\item[{\rm (i)}] $(\Sigma,\alpha)$ is index-positive, and the mean Euler characteristic of its filling is a half-integer.

\smallskip
\item[{\rm (ii)}]
If, in addition, $(\Sigma,\alpha)$ is a rational homo\-logy sphere,
then the mean Euler characteristic of its filling equals~$\frac{(-1)^{n+1}}{2}$.
\end{itemize}
}

\medskip
Given positive integers $a_0,\ldots,a_n$ one can define a Brieskorn manifold $\Sigma(a_0,\ldots,a_n)$
as the link of a certain singularity.
Such a Brieskorn manifold $\Sigma(a_0,\ldots,a_n)$ is said to be a \emph{non-trivial Brieskorn sphere}
if $a_i \neq 1$ for all~$i$.
If $a_0,\ldots,a_n$ are pairwise relatively prime, and if $n>2$,
then $\Sigma(a_0,\ldots,a_n)$ is homeomorphic to~$S^{2n-1}$.
Brieskorn manifolds carry a natural contact structure.
A trivial Brieskorn manifold is a round sphere with its standard contact structure in~$\R^{2n}$,
and hence admits a displaceable exact contact embedding.

\medskip
\textbf{Corollary~B. }
\emph{
A non-trivial Brieskorn sphere $\Sigma(a_0,\ldots,a_n)$ of dimension
at least~$5$ whose exponents are pairwise relatively prime
does not admit a displaceable exact contact embedding.
In particular, it does not admit an exact contact embedding
into a subcritical Stein manifold
whose first Chern class vanishes.}

The restriction to manifolds of dimension at least~$5$ comes from the following observations.

\begin{remark}
In dimension $3$, non-trivial Brieskorn spheres are not simply-connected.
Hence Reeb orbits in these manifolds can become contractible in the filling even
if they are not contractible in the Brieskorn manifold.
The mean Euler characteristic of symplectic homology, on the other hand, counts Reeb orbits that are contractible in the filling,
so it cannot be determined by just considering the contact manifold by itself.
\end{remark}

We conclude this introduction with a few open problems.

\smallskip
1.
Does Corollary~B still hold true if we drop the convexity assumption on the target manifold $(V,\lambda)$,
or the assumption that its first Chern class vanishes?

\smallskip
2.
Ritter proved in \cite{R} that the displaceability of $\Sigma$ implies the vanishing of
the symplectic homology of the filling~$W$.
It is conceivable that then in fact the equivariant symplectic homology
of~$W$ vanishes.
This would imply that the assumption in Theorem~A
that $(\Sigma, \alpha)$ is index-definite can be omitted,
see Remark~\ref{rem:definite}.

\section[The mean Euler characteristic]{The mean Euler
characteristic} \label{s:mean}

Assume that $(W,\lambda)$ is a compact exact symplectic manifold,
i.e.\ $\omega=d\lambda$ is a symplectic form on~$W$, with convex
boundary $\Sigma=\partial W$.
We assume throughout that
the first Chern class $c_1(W)$ of $(W,d\lambda)$ vanishes on $\pi_2(W)$, and that
$\Sigma$ is simply connected.
For $i \in \mathbb{Z}$ we denote by
$$
b_i(W) \,=\, \dim \big(SH^{S^1,+}_i(W;\mathbb{Q})\big)
$$
the $i$-th Betti number of the positive part of the equivariant
symplectic homology of~$W$
(as defined in~\cite{BO,V}).

For later use, we shall call a homology $H_*(C_*,\partial)$ \emph{index-positive} if there exists $N$ such that $H_i(C_*,\partial)=0$ for all $i<N$.
Note here that if $(\Sigma,\alpha)=\partial(W,d\lambda)$ is index-positive in the previously defined sense, then $SH^{S^1,+}_*(W)$ is index-positive in the homological sense.
The notions index-negative and index-definite are defined on homology level in a similar way.

\begin{definition}
$W$ is called \emph{homologically bounded} if the Betti numbers~$b_i(W)$ are uniformly bounded.
\end{definition}

If $W$ is homologically bounded we define its \emph{mean Euler characteristic} as
$$
\chi_m(W) \,=\, \lim_{N \to \infty} \frac{1}{N} \sum_{i=-N}^N (-1)^i b_i(W).
$$
The uniform bound on the Betti numbers implies that the limit exists.

\medskip
Now assume that $(\Sigma,\lambda)$ is a contact manifold with the property that
all closed Reeb orbits are non-degenerate.
We recall that a closed Reeb orbit $\gamma$ is called~\emph{bad} if
it is the $m$-fold cover of a Reeb orbit~$\gamma'$ and the difference of Conley--Zehnder indices
$\mu(\gamma)-\mu(\gamma')$ is odd.
A closed Reeb orbit which is not bad is called~\emph{good}.

\begin{definition}
$\Sigma$ is called \emph{dynamically bounded} if there exists a uniform bound for the number of
good closed Reeb orbits of Conley--Zehnder index~$i$ for every $i \in \mathbb{Z}$.
\end{definition}

We denote by $\mathfrak{G}_N$ the set of good closed Reeb orbits of Conley--Zehnder index
lying between~$-N$ and~$N$.
If $\Sigma$ is dynamically bounded, we define its mean Euler characteristic by
$$
\chi_m(\Sigma) \,=\,
\lim_{N \to \infty} \frac{1}{N} \sum_{\gamma \in \mathfrak{G}_N} (-1)^{\mu(\gamma)}.
$$

\begin{remark}
Ginzburg and Kerman, \cite{GK}, define the positive and negative part of the mean Euler characteristic of contact homology by summing over all positive and all negative degrees, respectively.
Their mean Euler characteristic is half of the one we define.
\end{remark}

If $W$ is a compact exact symplectic manifold, we say that $W$ is dynamically bounded
if its boundary $\Sigma=\partial W$ is dynamically bounded.

\begin{thm}\label{me}
Assume that $W$ is dynamically bounded. Then it is homologically bounded and
$$
\chi_m(\partial W) \,=\, \chi_m(W).
$$
\end{thm}

\begin{cor}
\label{cor:mec_invariant}
If $W$ is dynamically bounded, then its mean Euler characteristic is independent of the filling.
\end{cor}

\begin{remark}
Since the generators of the positive part of equivariant symplectic homology and contact homology
are the same, the mean Euler characteristic can also be expressed in terms of contact homology data.
This was done in the original definition in~\cite{vK}.
Note, however, that the degree of a Reeb orbit~$\gamma$ in contact homology is defined as $\mu_{CZ}(\gamma)+n-3$ if the dimension of the contact manifold is~$2n-1$.
This can result in a sign difference for the mean Euler characteristic.
\end{remark}

\textbf{Proof of Theorem~\ref{me}:}
If $\Gamma$ denotes the set of all closed Reeb orbits on $\Sigma=\partial W$,
then the critical manifold $\mathfrak{C}$ for the positive equivariant part of the
action functional of classical mechanics is given by
$$
\mathfrak{C} \,=\, \bigcup_{\gamma \in \Gamma} \gamma \times_{S^1} ES^1.
$$
If $\gamma$ is a $k$-fold cover of a simple Reeb orbit,
then the isotropy group of the action of~$S^1$ on~$\gamma$ is~$\mathbb{Z}_k$.
Therefore,
$$
\gamma \times_{S^1} ES^1 \,=\, B\mathbb{Z}_k
$$
is the infinite dimensional lens space.
The Morse--Bott spectral sequence, see \cite[Section 7.2.2]{FOOO},
tells us that there exists a spectral sequence converging to $SH_*^{S^1,+}(W;\mathbb{Q})$,
whose second page is given by
$$
E^2_{j,i} \,=\, \bigoplus_{\substack{\gamma \in \Gamma \\ \mu(\gamma)=i}}
H_j(\gamma \times_{S^1} ES^1 ;\mathcal{O}_\gamma).
$$
The twist bundle
$\mathcal{O}_\gamma$ is trivial if $\gamma$ is good, and equals the orientation bundle of the lens space
if $\gamma$ is bad, see \cite{BO1,BO,V1}.
The homology of an infinite dimensional lens space with rational coefficients equals $\mathbb{Q}$ in degree zero and vanishes otherwise.
Its homology with coefficients twisted by the orientation bundle is trivial.
Therefore the second page of the Morse--Bott spectral sequence simplifies to
$$
E^2_{j,i} \,=\, \bigoplus_{\substack{\gamma \in \mathfrak{G} \\ \mu(\gamma)=i}}
\mathbb{Q}
$$
where $\mathfrak{G} \subset \Gamma$ are the good closed Reeb orbits.
We conclude that the mean Euler characteristic of~$E^2$ coincides with $\chi_m(\Sigma)$.
Since the Euler characteristic is unchanged if we pass to homology, we deduce that $\chi_m(\Sigma)$
equals~$\chi_m(W)$.
This finishes the proof of the theorem.
\hfill $\square$

\medskip
In our application to Brieskorn manifolds,
we will compute the mean Euler characteristic for a contact form of Morse--Bott type.
Brieskorn manifolds can be thought of as Boothby--Wang orbibundles over symplectic orbifolds, since there is a contact form for which all Reeb orbits are periodic.
For such special contact manifolds the mean Euler characteristic has a particularly simple form.

We start with introducing some notation to state the result.
Consider a contact manifold $(\Sigma,\alpha)$ with Morse--Bott contact form $\alpha$ having only finitely many orbit spaces, so that we have an $S^1$-action on $\Sigma$.
Denote the periods by $T_1< \ldots< T_k$, so all $T_i$ divide~$T_k$.
Denote the subspace consisting of points on periodic Reeb orbits with period $T_i$ in $\Sigma$ by~$N_{T_i}$.

\begin{lemma}
\label{lemma:H^1_trivial}
If $H^1(N_{T_i};\Z_2)=0$,
then $H^1(N_{T_i} \times_{S^1} ES^1; \Z_2)=0$.
\end{lemma}

\begin{proof}
Consider the Leray spectral sequence for $N_{T_i} \times_{S^1} ES^1$
as a fibration over~$\C P^\infty$.
As $\pi_1(\C P^\infty)=0$, the Leray spectral sequence with $\Z_2$-coefficients converges
to the cohomology of $N_{T_i} \times_{S^1} ES^1$.
The $E_2$-page is given by $E_2^{pq} = H^p (\C P^\infty; H^q(N_{T_i};\Z_2) )$.
Since $H^1(N_{T_i};\Z_2)=0$ by assumption, there are no degree~$1$-terms on $E_2$.
Hence there are no degree~$1$-terms in $E_\infty$ either,
and $H^1(N_{T_i} \times_{S^1} ES^1; \Z_2)=0$.
\end{proof}

Finally we introduce the function
$$
\phi_{T_i;T_{i+1},\ldots,T_k} \,=\,
\# \{ a \in \N \mid aT_i < T_k \text{ and } a T_i \notin T_j \N \text{ for } j=i+1,\ldots, k \}
.
$$

\begin{proposition}
\label{prop:mean_euler_S^1-orbibundle}
Let $(\Sigma,\alpha)$ be a contact manifold as above and assume that it admits an exact filling $(W,d\lambda)$.
Suppose that $c_1(\xi=\ker \alpha)=0$, so that the Maslov index is well-defined.
Let $\mu_P:=\mu(\Sigma)$ be the Maslov index of a principal orbit of the Reeb action.
Assume that $H^1(N_T\times_{S^1} ES^1; \Z_2)=0$ for all $N_T$ and that there are no bad orbits.

If $\mu_P\neq 0$ then the following hold.
\begin{itemize}
\item $(\Sigma,\alpha)$ is homologically bounded.
\item $(\Sigma,\alpha)$ is index-positive if $\mu_P>0$ and index-negative if $\mu_P<0$.
\item The mean Euler characteristic satisfies the following formula,
$$
\chi_m(W)=\frac{\sum_{i=1}^k (-1)^{\mu(S_{T_i})-\frac{1}{2}\dim S_{T_i} } \phi_{T_i;T_{i+1},\ldots T_k} \chi^{S^1}(N_T)}{|\mu_P|}.
$$
\end{itemize}
\end{proposition}

Here $\chi^{S^1}(N_T)$ denotes the Euler characteristic of the $S^1$-equivariant homology of the $S^1$-manifold~$N_T$.

\begin{proof}
We use the notation
$$
H^{S^1}_{p}(N_T;\Q) \,:=\, H_{p}(N_T\times_{S^1}ES^1;\Q) .
$$
As before, there is the Morse--Bott spectral sequence converging to $SH_*^{S^1,+}(W;\Q)$.
The second page is given by
$$
E^2_{pq}=\bigoplus_{\substack{N_T \\ \mu(S_T)-\frac{1}{2}\dim S_T=q}} H^{S^1}_{p}(N_T;\Q).
$$
Indeed, the coefficient ring is not twisted as $H^1(N_T\times_{S^1} ES^1;\Z_2)=0$.

The period of a principal orbit is $T_k$, so we have $\phi^R_{T_k}=\id$.
Since the Robbin--Salamon version of the Maslov index is additive under concatenations,
it follows that for any set of periodic orbits $N_T$ with return time $T>T_k$ we have
$$
\mu(N_T)=\mu(N_{T_k})+\mu(N_{T-T_k}).
$$
It follows that the $E^2$-page is periodic in the $q$-direction with period $|\mu(N_{T_k})|=|\mu_P|$ (as $N_{T_k}=\Sigma)$.
Since we have assumed that $\mu_P\neq 0$, we see that $SH^{S^1,+}(W)$ is homologically bounded.

Moreover, by the definition of the Maslov index~$\mu_P$,
the sign of $\mu_P$ determines whether $(\Sigma,\alpha)$ is index-positive or index-negative.

Finally, the mean Euler characteristic can be obtained by summing all contributions in one period and dividing by the period.
This gives
$$
\chi_m(W)=\frac{\sum_{T \leq T_k} (-1)^{\mu(S_T)-\frac{1}{2}\dim S_T }\chi^{S^1}(N_T)}{|\mu_P|}.
$$
Now observe that the definition of the functions $\phi_{T_i;T_{i+1},\ldots T_{k}}$ is such that it counts how often multiple covers of a set of periodic orbits~$N_{T_i}$ appear in one period of the $E^2$-page without being contained in a larger orbit space.
We thus obtain the above formula.
\end{proof}

\begin{remark}
This proposition is a generalization of \cite{E}[Example 8.2], and Espina's methods could also be used to show the above.
\end{remark}

\section[Proof of Theorem~A]{Proof of Theorem~A}

In the first two paragraphs of this section we prove three general statements
that in particular imply assertion~(i) of Theorem~A.
In paragraphs~\ref{ss:dis} and \ref{ss:rat} we then work out
the situation for rational homology spheres.

\subsection{Two general statements} \label{ss:two}
\begin{proposition}
\label{prop:equivariant}
Assume that $(\Sigma,\xi)$
is a $(2n-1)$-dimensional simply-connected contact manifold admitting a displaceable exact contact embedding into $(V,d\lambda)$.
Denote the compact component of $V \setminus \Sigma$ by~$W$.
Suppose furthermore that $(\Sigma,\alpha)$ is index-positive.
Then
$$
SH^{S^1,+}_{*}(W) \,\cong\, H_{*+n-1}^{S^1}(W,\Sigma).
$$
\end{proposition}

\begin{corollary}
\label{cor:mean_euler}

Under the assumptions of Proposition~\ref{prop:equivariant},
$$
\chi_m(W) \,=\, (-1)^{n+1}\, \frac{\chi(W,\Sigma)}{2}.
$$
\end{corollary}

\begin{proof}[Proof of Proposition~\ref{prop:equivariant}]
Consider the $S^1$-equivariant version of the Viterbo long exact sequence,
$$
\ldots \longrightarrow H^{S^1}_{*+n}(W,\Sigma) \longrightarrow SH^{S^1}_*(W)  \longrightarrow SH^{S^1,+}_{*}(W) \longrightarrow H^{S^1}_{*+n-1}(W,\Sigma) \longrightarrow \ldots
$$
from~\cite{V,BO}.
By assumption $SH^{S^1,+}_*(W)$ is index-positive.
The group homology $H^{S^1}_{*}(W,\Sigma)$ is also index-positive, so we conclude that $SH^{S^1}_*(W)$ must be index-positive as this group is sandwiched between $0$'s for sufficiently negative~$*$.

By Ritter's theorem \cite[Theorem 97]{R} displaceability of~$\Sigma$
implies $SH_*(W)=0$.
The Gysin sequence for equivariant and non-equivariant symplectic homology from~\cite{BO} reads
$$
\ldots \longrightarrow SH_*(W) \longrightarrow SH^{S^1}_*(W)  \stackrel{D_*}{\longrightarrow} SH^{S^1}_{*-2}(W) \longrightarrow SH_{*-1} (W) \longrightarrow \ldots
$$
so all maps $D_*$ are isomorphisms.
Since we just showed that $SH^{S^1}_*(W)$ is index-positive, it must vanish in all degrees.

Finally consider the equivariant version of the Viterbo sequence once again.
Since $SH^{S^1}_*(W)$ vanishes, it follows that
$$
H_{*+n}^{S^1}(W,\Sigma) \,\cong\, SH^{S^1,+}_{*+1}(W).
$$
\end{proof}

Corollary~\ref{cor:mean_euler}
follows from Proposition~\ref{prop:equivariant}
by observing that $H^{S^1}_*(W,\Sigma) \cong H_*(W,\Sigma)\otimes H_*(\C P^\infty)$,
since the $S^1$-action on $(W,\Sigma)$ trivial
(by construction of Viterbo's long exact sequence).
In other words, $H^{S^1}_*(W,\Sigma)$ consists of infinitely many copies of $H_*(W,\Sigma)$ which are degree-shifted by $0,2,4,\ldots$

\begin{remark} \label{rem:definite}
It is conceivable that the displaceability of~$W$ implies that
$SH^{S^1}_*(W)$ vanishes.
The conclusion of Proposition~\ref{prop:equivariant},
without the assumption that $(\Sigma, \alpha)$ is index-positive,
would then follow at once from Viterbo's $S^1$-equivariant long exact sequence.
Hence, the assumption in Theorem~A
that $(\Sigma, \alpha)$ is index-definite could be omitted.
\end{remark}

\subsection{Index-positivity}

\begin{lemma}
Assume that $(\Sigma,\xi)$
is a $(2n-1)$-dimensional simply-connected contact manifold admitting a displaceable exact contact embedding into $(V,d\lambda)$.
Denote the compact component of $V \setminus \Sigma$ by $W$.
Suppose furthermore that $(\Sigma,\xi=\ker \alpha)$ is index-definite.
Then $(\Sigma,\alpha)$ is index-positive.
\end{lemma}

\begin{proof}
Again by Ritter's theorem \cite[Theorem 97]{R} we conclude that $SH_*(W)=0$.
Hence the Viterbo long exact sequence
from \cite{V,BO}
reduces to
$$
\ldots \longrightarrow 0 \longrightarrow SH^+_*(W) \stackrel{\cong}{\longrightarrow} H_{*+n-1}(W,\Sigma) \longrightarrow 0 \longrightarrow \ldots,
$$
so we see that $SH^+_{n+1}(W)\cong H_{2n}(W,\Sigma)\cong H^0(W)\neq 0$.

Now suppose that $(\Sigma,\alpha)$ is index-negative.
On one hand, our previous observation shows that there is a generator of degree~$n+1$.
On the other hand, if $\alpha$ is a non-degenerate contact form,
then the iteration formula~\eqref{eq:iteration_and_mean_index}
tells us that an $N$-fold cover of a Reeb orbit~$\gamma$ satisfies
$$
|\mu_{CZ}(\gamma^N) - N\Delta(\gamma) | \,\leq\, n-1,
$$
where $\Delta(\gamma)$ denotes the mean index of the Reeb orbit $\gamma$.
Since $(\Sigma,\alpha)$ is index-negative, $\Delta(\gamma)<0$, so $\mu_{CZ}(\gamma^N)< n-1$.
In particular, no generator of $SH^{+}_{n+1}(W)$ can be realized by a Reeb orbit.
This contradiction shows that $(\Sigma,\alpha)$ must be index-positive.
\end{proof}


\subsection{Displaceability and splitting the sequence of the pair}
\label{ss:dis}

In the following lemma, $(V,\Omega)$ is a connected manifold endowed with a volume form,
and $W \subset V$ is a compact connected submanifold with connected boundary of the same
dimension as~$V$ with the property that the volume of the complement
of~$W$ in~$V$ is infinite.
We say that the hypersurface $\Sigma = \partial W \subset V$ is \emph{volume preserving displaceable}
if there exists a compactly supported smooth family of volume
preserving vector fields $X_t$, $t \in [0,1]$, on~$V$
such that the time-$1$ map~$\phi$ of its flow satisfies
$\phi(\Sigma) \cap \Sigma =\emptyset$.

\begin{lemma}\label{disp}
Assume that $\Sigma=\partial W$ is volume preserving displaceable in~$V$.
Then the projection homomorphism $p_*\colon H_*(W;\mathbb{Q}) \to H_*(W,\Sigma;\mathbb{Q})$ vanishes.
\end{lemma}

\textbf{Proof: }
We prove the lemma in two steps. For the first step we need the assumption about volume preservation.
\\ \\
\textbf{Step~1: }
\emph{The volume preserving diffeomorphism $\phi$ displacing $\Sigma$ displaces the whole filling,
i.e.~ $\phi(W) \cap W=\emptyset$.}
\\ \\
We divide the proof of Step~1 into three substeps.
\\ \\
\textbf{Step~1a: } \emph{There exists a point $x \in W$ such that $\phi(x) \notin W$.}
\\ \\
We argue by contradiction and assume that $\phi(W) \subset W$.
In particular, the restriction of~$\phi$ to~$W$
gives a diffeomorphism between the two manifolds with boundary~$W$ and $\phi(W) \subset W$.
Therefore, if $y \in W$ satisfies $\phi(y) \in \partial W$, it follows that $y \in \partial W$.
We conclude
$$\phi(W) \cap \partial W \,\subset\, \phi(\partial W).$$

Since $\phi$ displaces the boundary from itself, we obtain
$$\phi(W) \cap \partial W = \emptyset .$$
Denoting by $\mathrm{int}$ the interior of a set, we can write this equivalently as
$$\phi(W) \,\subset\, \mathrm{int}(W).$$
Hence $\phi(W)$ is a strict subset of~$W$.
Since $W$ is compact, its volume is finite.
Therefore, the volume of~$\phi(W)$ is strictly less than the volume of~$W$.
This contradicts the fact that $\phi$ is volume preserving.
Therefore the assertion of Step~1a has to hold true.
\\ \\
\textbf{Step~1b: } $\phi(\partial W) \subset W^c$.
\\ \\
Since $\phi$ displaces $\partial W$ from itself, we have
$\partial W \subset \mathrm{int}(W) \cup W^c$. Since $\partial W$ is connected by assumption,
we either have $\phi(\partial W) \subset \mathrm{int}(W)$ or $\phi(\partial W) \subset W^c$.
Therefore it suffices to show that $\phi(\partial W) \cap W^c$ is not empty.
Since $W^c$ has infinite volume but $\phi(W)$ has finite volume by assumption,
we conclude that there exists a point $y_0 \in W^c$
such that $y_0 \notin \phi(W)$. Step~1a implies the existence of a
point $y_1 \in W^c$ satisfying $y_1 \in \phi(W)$.
Since $V,W$, and $\partial W$ are connected by assumption, we obtain from the
Mayer--Vietoris long exact sequence that $W^c$ is connected as well.
Therefore there exists a path $y \in C^0([0,1],W^c)$ satisfying $y(0)=y_0$ and $y(1)=y_1$.
Since $W^c$ is Hausdorff, there exists $t \in (0,1)$ such that
$y(t) \in \partial (\phi(W)) = \phi(\partial W)$.
Therefore $\phi(\partial W) \cap W^c$ is not empty, which finishes the proof of Step~1b.
\\ \\
\textbf{Step~1c: } \emph{We prove Step~1.}
\\ \\
We assume by contradiction that there exists a point $x_0 \in W \cap \phi(W)$.
By Step~1a and the fact that $\phi$ is volume preserving,
we conclude that $W$ cannot be a subset of~$\phi(W)$.
Therefore there has to exist a point $x_1 \in W \cap (\phi(W))^c$ as well.
Since $W$ is connected by assumption, there exists a path $x \in C^0([0,1],W)$
satisfying $x(0)=x_0$ and $x(1)=x_1$. As in Step~1b there has to exist $t \in (0,1)$
such that $x(t) \in \phi(\partial W)$. But this contradicts the assertion of Step~1b.
The proof of Step~1 is complete.
\\ \\
\textbf{Step~2: } \emph{If a diffeomorphism $\phi$ isotopic to the identity satisfies
$\phi(W) \cap W =\emptyset$, then the projection homomorphism
$p_* \colon H_*(W;\mathbb{Q}) \to H_*(W,\partial W; \mathbb{Q})$ vanishes.}
\\ \\
We prove the dual version in de~Rham cohomology, i.e.~we show that
the inclusion homomorphism from the compactly supported de~Rham cohomology of~$W$
to the de~Rham cohomology of~$W$ vanishes.
To see this, pick $\omega \in \Omega^k(W)$ which is compactly supported and closed.
We show that there exists $\eta \in \Omega_{k-1}(W)$ not necessarily compactly supported
such that $\omega=d\eta$.
Since $\omega$ is compactly supported, we can extend it trivially to a
closed $k$-form on~$V$, which we refer to as~$\widetilde{\omega}$.
Since $\phi$ is isotopic to the identity, we have $\phi=\phi^1$ for a
flow $\{\phi^t\}_{t \in [0,1]}$ generated by a time dependent vector field~$X_t$.
By the Cartan formula and the fact that $\widetilde{\omega}$ is closed we obtain
\begin{eqnarray*}
\frac{d}{dt}\big(\phi^t\big)^*\widetilde{\omega} &=&
L_{X_t}\big(\phi^t\big)^*\widetilde{\omega}\\
&=&\big(d i_{X_t}+i_{X_t}
d\big)\big(\phi^t\big)^*\widetilde{\omega}\\
&=&d i_{X_t}\big(\phi^t\big)^*\widetilde{\omega}.
\end{eqnarray*}
We define a $(k-1)$-form on $V$ by the formula
$$\widetilde{\eta} \,=\, -\int_0^1
i_{X_t}\big(\phi^t\big)^*\widetilde{\omega} .
$$
By the previous computation we get
$$\widetilde{\omega}-\phi^*\widetilde{\omega} \,=\, d\widetilde{\eta}.$$
Now set
$$\eta \,=\, \widetilde{\eta}|_{W} \,\in\, \Omega^{k-1}(W).$$
Since $\phi$ displaces $W$ we obtain
$$\omega \,=\, d\eta.$$
This finishes the proof of Step~2 and hence of Lemma~\ref{disp}.
\hfill $\square$

\subsection{Rational homology spheres and completion of the proof of Theorem~A~(ii)}
\label{ss:rat}

In the case of rational homology spheres, the homology of the filling is completely determined:
\begin{lemma}\label{homi}
Suppose $(\Sigma,\xi)$
is a $(2n-1)$-dimensional simply-connected rational homology sphere
admitting a displaceable exact contact embedding into $(V,d\lambda)$.
Let $W$ denote the compact component of $V \setminus \Sigma$. Then
$$
H_*(W,\Sigma;\mathbb{Q}) \,=\,
\left\{\begin{array}{cl}
\mathbb{Q} & \mbox{if }\, *=2n, \\
\{0\} & \textrm{else.}
\end{array}\right.
$$
\end{lemma}

\begin{proof}
We can assume that $V$ has infinite volume.
Indeed, if $V$ has finite volume, we choose a compact convex manifold $V_k$
in the exhaustion of~$V$ such that $W \subset V_k$, and replace $V$
by the manifold $\widehat V$ obtained by attaching cylindrical ends
to the boundary of~$V_k$.
Notice that $\widehat V$ is also an exact convex manifold whose first Chern class
vanishes on $\pi_2(\widehat V)$.

In view of Lemma~\ref{disp} the long exact homology sequence for the pair $(W,\Sigma)$
splits for every $k \in \mathbb{Z}$ into short exact sequences
$$
0 \longrightarrow H_k(W,\Sigma;\mathbb{Q})
\stackrel{\partial}\longrightarrow H_{k-1}(\Sigma;\mathbb{Q})
\stackrel{i_*} \longrightarrow  H_{k-1}(W;\mathbb{Q})
\longrightarrow 0 .
$$
By using the fact that $\Sigma$ is a rational homology sphere as well as
$H_0(W;\mathbb{Q}) = \mathbb{Q}$ we conclude that
$H_*(W,\Sigma;\mathbb{Q})=\{0\}$ for $* \neq 2n$.
Since $H^{2n}(W,\Sigma;\mathbb{Q})$ is Poincar\'e dual to
$H_0(W;\mathbb{Q})$, the result for $*=2n$ also follows.
\end{proof}

Lemma~\ref{homi} shows that the Euler characteristic of the relative homology
is given by
$$
\chi(W,\Sigma)=1.
$$
Assertion~(ii) of Theorem~A follows from this and Corollary~\ref{cor:mean_euler}.

\section[Brieskorn manifolds]{Brieskorn manifolds} \label{s:Brieskorn}
Choose positive integers $a_0,\ldots,a_n$.
The Brieskorn variety $V(a_0,\ldots,a_n)$ is defined as the following subvariety of $\C^{n+1}$,
$$
V_\epsilon(a_0,\ldots,a_n) \,=\,
\biggl\{
(z_0,\ldots,z_n)\in \C^{n+1}~|~\sum_{i=0}^n z_i^{a_i}=\epsilon
\biggr\} .
$$
For $\epsilon=0$, this variety is singular unless one of the exponents~$a_i$ is equal to~$1$.
For $\epsilon \neq 0$,
we have a complex submanifold of~$\C^{n+1}$.

Given a Brieskorn variety $V_0(a_0,\ldots,a_n)$ we define the Brieskorn manifold as
$$
\Sigma(a_0,\ldots,a_n)
\,:=\, V_0(a_0,\ldots,a_n) \cap S^{2n+1}_R,
$$
where $S^{2n+1}_R$ is the sphere of radius $R>0$ in~$\C^{n+1}$.
For the diffeomorphism type, the precise value of~$R$ does not matter.
Brieskorn manifolds carry a natural contact structure, which comes from the following construction.
\begin{lemma}
Let $(W,i)$ be a complex variety together with a function $f$ that is plurisubharmonic away from singular points.
Then regular level sets $M=f^{-1}(c)$ carry a contact structure
$\xi = TM \cap i \,TM = \ker (-df \circ i)|_M$.
\end{lemma}

Applying this lemma with the plurisubharmonic function $f=\sum_j\frac{a_j}{8}|z_j|^2$
we obtain the particularly nice contact form
$$
\alpha \,=\, \frac{i}{8} \sum_j a_j \left(z_j d\bar z_j-\bar z_j dz_j \right)
$$
for this natural contact structure.
Its Reeb vector field at radius~$R=1$ is given by
$$
R_\alpha \,=\, 4i \sum_j \frac{1}{a_j} (z_j \partial_{z_j}- \bar z_j \partial_{\bar z_j}) .
$$
The Reeb flow therefore is
$$
\phi^{R_\alpha}_t (z_0,\ldots,z_n) \,=\, \left( e^{4it/a_0}z_0,\ldots,e^{4it/a_n}z_n \right).
$$
We thus see that all Reeb orbits are periodic.
This allows us to interpret Brieskorn manifolds as Boothby--Wang bundles over symplectic orbifolds.

\begin{proposition}
\label{prop:homology_sphere}
Brieskorn manifolds admit a Stein filling, and their contactomorphism type does not depend on the radius~$R$ of the sphere used to define them.
\end{proposition}

Indeed, by definition, Brieskorn manifolds are singularly fillable.
One can smoothen this filling by taking $\epsilon \neq 0$, and consider $V_\epsilon$ rather than~$V_0$.
The resulting contact structure is contactomorphic by Gray stability.
Furthermore, $V_\epsilon$ gives then the Stein filling.
Gray stability can also be used to show independence of the radius~$R$,
see also Theorem~7.1.2 from~\cite{G}.

\subsection{Brieskorn manifolds and homology spheres}
Let us start by citing some theorems from~\cite{HM}.
This book gives precise conditions for Brieskorn manifolds to be
integral homology spheres.
However, we shall restrict ourselves to the following case.

\begin{proposition} \label{s5:prop1}
If $a_0,\ldots,a_n$ are pairwise relatively prime, then $\Sigma(a_0,\ldots,a_n)$ is an integral homology sphere.
\end{proposition}

Furthermore, higher dimensional Brieskorn manifolds, i.e.~$\dim \Sigma>3$, are always simply-connected,
so we in fact find
\begin{theorem} \label{s5:t}
If $a_0,\ldots,a_n$ are pairwise relatively prime, and if $n>2$,
then $\Sigma(a_0,\ldots,a_n)$ is homeomorphic to~$S^{2n-1}$.
\end{theorem}

\begin{remark}
\label{rem:exponent=1}
If one of the exponents $a_j$ is equal to~$1$, then the resulting Brieskorn manifold $(\Sigma(a_0,\ldots,a_n),\alpha)$ is contactomorphic to the standard sphere $(S^{2n-1},\alpha_0)$.
Indeed, in this case the Brieskorn variety~$V_\epsilon(a_0,\ldots,a_n)$
is biholomorphic to~$\C^{n}$,
as we can regard the variety as a graph.
\end{remark}

\subsection{Formula for the mean Euler characteristic for Brieskorn manifolds}
We can think of Brieskorn manifolds as Boothby--Wang orbibundles over symplectic orbifolds.
However, all the essential data is contained in the $S^1$-equivariant homology groups associated with the Reeb action.
The following lemma will hence be useful.

\begin{lemma}
\label{lemma:S^1-euler_characteristic}
Let $N$ be
a rational homology sphere of dimension $2n+1$ with a fixed-point free $S^1$-action $N \times S^1 \to N$.
Then
$$
H_*^{S^1}(N;\Q) \,\cong\, H_*(\C \P^n;\Q).
$$
In particular,
$$
\chi^{S^1}(N) \,=\, n+1.
$$
\end{lemma}

\begin{proof}
Note that $N \times ES^1$ carries a free $S^1$-action,
so we can think of $N \times ES^1$ as an $S^1$-bundle over $N \times_{S^1} ES^1$.
We consider the Gysin sequence for this space with $\Q$-coefficients.
Since $N$ is
a rational homology sphere of dimension $2n+1$ and $ES^1$ is contractible,
all homology groups of $N \times ES^1$ except in dimension~$0$ and $2n+1$ vanish.
Hence the Gysin sequence reduces to
$$
\underset{\cong 0}{H_*(N)}
\stackrel{\pi_*}{\longrightarrow} \underset{= H^{S^1}_*(N)}{H_*(N\times_{S^1}ES^1)}\stackrel{\cap e}{\longrightarrow}\underset{= H^{S^1}_{*-2}(N)}{H_{*-2}(N\times_{S^1}ES^1) }
\longrightarrow
\underset{\cong 0}{H_{*-1}(N)}
$$
for $1<*<2n+1$.
This shows that $H_*^{S^1}(N;\Q) \cong H_*(\C \P^n;\Q)$ for $*<2n+1$.
To see that there are no other terms, we shall argue that $H_*^{S^1}(N;\Q)$ is bounded.
For this, choose an $S^1$-equivariant Morse--Bott function $f \colon N \to \R$,
see \cite[Lemma 4.8]{W} for the existence of such a function.
Define a Morse--Bott function
\begin{align*}
\tilde f \colon N \times_{S^1} ES^{1} & \longrightarrow \R \\
[x,v] & \longmapsto f(x).
\end{align*}
Consider the Morse--Bott spectral sequence for $H_*(N \times_{S^1} ES^{1};\Q)$
with respect to the Morse--Bott function~$\tilde f$.
Its $E^2$-page is given by $E^2_{pq} = H_q(R_p;\Q)$, where $R_p$ are the critical manifolds of $\tilde f$ with index~$p$.
Again, by \cite[Section 7.2.2]{FOOO} this sequence converges to $H_*(N \times_{S^1} ES^{1};\Q)$.
Note that the critical manifolds form infinite-dimensional lens spaces, so $H_q(R_p;\Q) \cong \Q$
if $q=0$ and $0$ otherwise.
Since there are only finitely many critical manifolds (because $N$ is compact),
it follows that $H_*^{S^1}(N;\Q)$ is bounded.

With this in mind, we reexamine the Gysin sequence.
Assume that $H_k^{S^1}(N;\Q)$ is non-zero for
some $k\geq 2n+1$.
Then $H_{k+2}^{S^1}(N;\Q)$ is non-zero either, etc.
Hence $H_*^{S^1}(N;\Q)$ is not bounded, which contradicts our previous term.
The lemma follows.
\end{proof}

\begin{remark}
Strictly speaking, $N \times_{S^1} ES^1$ has no manifold structure.
Recalling $ES^1 = S^\infty$,
we can, however, approximate this space by $N \times_{S^1} S^{2M+1}$ for large $M$.
For the latter space, the above argument works, and can be adapted to show triviality of
$H_i (N \times_{S^1} S^{2M+1};\Q)$ for $i \geq 2n+1$ and $i<2M$.
\end{remark}

\begin{proposition}
\label{prop:mean_euler_Brieskorn}
The Brieskorn manifold $\Sigma(a_0,\ldots,a_n)$ with its natural contact form~$\alpha$
is index-positive if $\sum_j \frac{1}{a_j} >1$, and index-negative if $\sum_j \frac{1}{a_j} <1$.
Furthermore, if the exponents $a_0,\ldots,a_n$ are pairwise relatively prime,
then the mean Euler characteristic of $\Sigma(a_0,\ldots,a_n)$ is given by
\begin{equation}
\label{eq:mean_euler}
\chi_m(\Sigma(a_0,\ldots,a_n),\alpha) \,=\,
(-1)^{n+1}\,
\frac{n+(n-1) \sum_{i_0}(a_{i_0}-1)+\ldots + 1 \sum_{i_0<\ldots< i_{n-2} }(a_{i_0}-1)\cdots (a_{i_{n-2}}-1)
}{
2 | (\sum_j a_0 \cdots \widehat{a_j}\cdots a_n)  -a_0\cdots a_n|
}
\end{equation}
\end{proposition}

\begin{proof}
The proof is a direct application of Proposition~\ref{prop:mean_euler_S^1-orbibundle}.
The principal orbits have period $a_0 \cdots a_n$.
Exceptional orbits have periods $a_0, \ldots, a_n,\, a_0 a_1, \ldots, a_{n-1}a_n, \ldots, a_1 \cdots a_n$.
Given a collection of exponents $I=\{ a_{i_1},\ldots,a_{i_k} \}\subset \{ a_0,\ldots, a_n \}$ we denote the associated subset of
periodic orbits with period $a_{i_1}\cdots a_{i_k}$ by $N_I$.

In~\cite{vK} the Maslov index of all periodic Reeb orbits is computed.
For the principal orbit, the result is
$$
\mu_P \,:=\, 2 \lcm_{i}a_i \left( \sum_j \frac{1}{a_j} -1 \right)
\,=\, 2 \sum_j a_0 \cdots \widehat{a_j}\cdots a_n   -a_0\cdots a_n.
$$
We check that the conditions of Proposition~\ref{prop:mean_euler_S^1-orbibundle} are satisfied.
By Proposition~\ref{s5:prop1} it follows that $H^1(N_I;\Z_2)=0$ if the index set~$I$ has more than~$2$ elements (i.e.~$\dim N_T>1$), so Lemma~\ref{lemma:H^1_trivial} applies.
Furthermore, the index computations in~\cite{vK} show that there are no bad orbits.

Hence Proposition~\ref{prop:mean_euler_S^1-orbibundle} applies, so $\Sigma(a_0,\ldots,a_n)$ is index-positive if $\sum_j \frac{1}{a_j} >1$ and index-negative if $\sum_j \frac{1}{a_j} <1$.
Furthermore, the $S^1$-equivariant Euler characteristics needed in Proposition~\ref{prop:mean_euler_S^1-orbibundle} are obtained from Lemma~\ref{lemma:S^1-euler_characteristic}.

The formula for the Maslov index of the exceptional orbits is slightly more complicated,
see Formula~(3.1) from~\cite{vK2},
but we only need to observe that the parity of $\mu(S_{T_i})-\frac 12 \dim S_{T_i}$
is the same as the one of~$n+1$.

We conclude the proof by determining the coefficients $\phi_{T_i;T_{i+1},\ldots,T_n}$.
We shall do this by counting how often multiple covers of an orbit space appear in one period.
The full orbit space $S_{\{a_0,\ldots,a_n\}}$ appears once.
The orbit space $S_{\{a_0,\ldots,a_{n-1}\}}$ appears $a_n$~times, but the last time it contributes,
it is part of $S_{\{a_0,\ldots,a_n\}}$, which we already considered.
Therefore $S_{\{a_0,\ldots,a_{n-1}\}}$ contributes $a_n-1$~times.
By downwards induction on the cardinality of~$I$, we conclude that $S_{ I}$ appears
$\prod_j (a_j-1) / \prod_{a\in I}(a-1)$ times in one period.
\end{proof}

\begin{remark}
The mean Euler characteristic of $S^1$-equivariant symplectic homology coincides with the mean Euler characteristic of contact homology.
This means that the above computation amounts to an application of the algorithm in~\cite{vK2}.
However, there are still many issues with the foundations of contact homology, so we shall not pursue this line of thought.
\end{remark}

\section{Proof of Corollary~B. }
We start by some general observations that will be needed in the proof.

For $n \in \mathbb{N}$ we define
$$f(n) \,=\, \sum_{j=0}^n (-1)^j(n-j){n+1 \choose j}.$$
We claim the following identity
\begin{equation}\label{euler}
f(n) \,=\, (-1)^{n+1}
\end{equation}

We prove \eqref{euler} by induction.
It holds that $f(1)=1$, and for the induction step we compute
\begin{eqnarray*}
f(n+1)&=&\sum_{j=0}^{n+1}(-1)^j(n+1-j){n+2 \choose j}\\
&=&\sum_{j=0}^n(-1)^j(n+1-j){n+2 \choose j}\\
&=&\sum_{j=0}^n(-1)^j(n+1-j)\Bigg({n+1 \choose j}+{n+1 \choose
j-1}\Bigg)\\
&=&\sum_{j=0}^n(-1)^j(n+1-j){n+1 \choose j}+
\sum_{j=0}^n(-1)^j(n+1-j){n+1 \choose j-1}\\
&=&\sum_{j=0}^n(-1)^j(n-j){n+1 \choose j}+ \sum_{j=0}^n(-1)^j {n+1
\choose j}+ \sum_{j=1}^n(-1)^j(n-(j-1)){n+1 \choose j-1}\\
&=&(-1)^{n+1}+ \sum_{j=0}^{n+1}(-1)^j {n+1 \choose j}-(-1)^{n+1} +
\sum_{j=0}^{n-1}(-1)^{j+1}(n-j){n+1 \choose j}\\
&=&(-1)^{n+1}+(1-1)^{n+1}-(-1)^{n+1}-(-1)^{n+1}\\
&=&-(-1)^{n+1}\\
&=&(-1)^{n+2}.
\end{eqnarray*}
This proves the induction step and hence \eqref{euler} follows.

Alternatively, we can compute
\begin{eqnarray*}
0 \,=\,
\frac{d}{dx}(-1+x)^{n+1}|_{x=1}&=&\sum_{j=0}^n(-1)^j(n+1-j)x^{n-j}{n+1 \choose j}|_{x=1} \\
&=& f(n)+\sum_{j=0}^n(-1)^j {n+1 \choose j}+(-1)^{n+1}{n+1 \choose n+1}-(-1)^{n+1} \\
&=& f(n)+(-1+1)^{n+1}-(-1)^{n+1}=f(n)-(-1)^{n+1}.
\end{eqnarray*}

\begin{proposition}
\label{prop:mean_euler=1}
Let $\Sigma(a_0,\ldots,a_n)$ be a Brieskorn manifold whose exponents are pairwise relatively prime.
Suppose that $\sum_j \frac{1}{a_j}>1$.
Then $\chi_m (\Sigma(a_0,\ldots,a_n),\alpha) = \frac{(-1)^{n+1}}{2}$ if and only if one of the exponents
is equal to~$1$.
\end{proposition}

\begin{proof}
The condition $\sum_j \frac{1}{a_j}>1$ implies that the denominator of~\eqref{eq:mean_euler}
(without $| \; |$)
is positive, so
$$
\chi_m (\Sigma(a_0,\ldots,a_n),\alpha) \,=\,  (-1)^{n+1} \,
\frac{n+(n-1)\sum_{i_0}(a_{i_0}-1)+\ldots
+\sum_{i_0<\ldots< i_{n-2} }(a_{i_0}-1)\cdots (a_{i_{n-2}}-1)
}{
2\left(  (\sum_j a_0 \cdots \widehat{a_j}\cdots a_n)   -a_0\cdots a_n \right)
}
.
$$
Let us now try to solve the equation $\chi_m= (-1)^{n+1}  \frac{1}{2}$.
We obtain
$$
n+(n-1)\sum_{i_0}(a_{i_0}-1)+\ldots
+\sum_{i_0<\ldots< i_{n-2} }(a_{i_0}-1)\cdots (a_{i_{n-2}}-1)
\,=\,
(\sum_j a_0 \cdots \widehat{a_j}\cdots a_n)   -a_0\cdots a_n .
$$
We multiply out all terms on the left hand side and organize them as linear combinations of elementary symmetric polynomials $e_d(a_0,\ldots,a_n)$ of degree~$d$, for $d=0,\ldots,n-2$.
Using Formula~\eqref{euler} repeatedly to obtain
$$
\sum_{k=0}^{n-2} (-1)^{n-2-k}
e_k(a_0,\ldots,a_n)
\,=\,
 e_{n-1}(a_0,\ldots,a_n)   -e_n(a_0,\ldots,a_n) .
$$
Moving all terms to the left hand side and collecting them yields the equation
$$
\prod_{j=0}^n(a_j-1)=0 ,
$$
which can only hold if one of the exponents is equal to~$1$.
\end{proof}

Observe that the remark after Theorem~\ref{s5:t} implies that the mean Euler characteristic has to be equal to~$  \frac{(-1)^{n+1}}{2}$ if one of the exponents equals~$1$.

\medskip
{\bf Proof of Corollary~B.}
Let $\Sigma(a_0,\ldots,a_n)$ be a Brieskorn manifold with pairwise relatively prime exponents $a_0,\ldots,a_n$.
If the exponents $a_0,\ldots,a_n$ satisfy $\sum_j \frac{1}{a_j}<1$, then Proposition~\ref{prop:mean_euler_Brieskorn} tells us that $(\Sigma(a_0,\ldots,a_n),\alpha)$ is index-negative.
Theorem~A implies that such manifolds do not admit a displaceable exact contact embedding.

If the exponents $a_0,\ldots,a_n$ are pairwise relatively prime, then $\sum_j \frac{1}{a_j} \neq 1$.
Indeed, suppose that $\sum_j \frac{1}{a_j}=1$.
Then
$$
\frac{1}{a_0} \,=\, 1-\sum_{j=1}^n \frac{1}{a_j} \,=\,
\frac{a_1 \cdots a_n - \sum_{j=1}^n a_1 \cdots \widehat{a_j} \cdots a_n}{a_1 \cdots a_n}.
$$
If we invert the left and right hand side, we see that $a_0$ divides $a_1\cdots a_n$,
which shows that $a_0,\ldots,a_n$ are not pairwise relatively prime.
This leaves the case that $\sum_j \frac{1}{a_j}>1$.
For this case, Proposition~\ref{prop:mean_euler=1} applies, so together with Theorem~A we conclude that non-trivial Brieskorn manifolds with pairwise relatively prime exponents do not admit exact displaceable contact embeddings.

{\bf Acknowledgment.} We thank Kai Cieliebak and Alex Oancea for
useful discussions which inspired Step~1 of the proof of
Lemma~\ref{disp}. The second author heartily thanks Seoul National
University for its warm hospitality during a beautiful week in
February~2011. The first author was partially supported by the Basic
Research fund 2010-0007669 funded by the Korean government, the
second author by SNF grant 200021-125352/1 and 
the third author by the New Faculty Research
Grant 0409-20100147 funded by the Korean government.


\begin{thebibliography}{999}
\bibitem{BO1}
F.~Bourgeois, A.~Oancea,
\emph{Symplectic homology, autonomous Hamiltonians, and Morse--Bott moduli spaces},
Duke Math.\,J \textbf{146} (2009) 71--174.

\bibitem{BO}
F.~Bourgeois, A.~Oancea,
\emph{The Gysin exact sequence for $S^1$-equivariant symplectic homology},
arXiv:0909.4526.

\bibitem{CF}
K.~Cieliebak, U.~Frauenfelder,
\emph{A Floer homology for exact contact embeddings},
Pacific J.\,Math.~\textbf{239} (2009) 251--316.

\bibitem{E}
J.~Espina,
\emph{On the mean Euler characteristic of contact manifolds},
arXiv:1011.4364

\bibitem{FOOO}
K.~Fukaya, Y.~Oh, H.~Ohta, K.~Ono,
\emph{Lagrangian intersection Floer Theory - Anomaly and Obstruction},
AMS/IP Studies in Advanced Mathematics vol 46.1 (2009).

\bibitem{G}
H.~Geiges,
\emph{An introduction to contact topology},
Cambridge Studies in Advanced Mathematics, 109.
Cambridge University Press, Cambridge, 2008.

\bibitem{GK}
V.~Ginzburg, E.~Kerman,
\emph{Homological resonances for Hamiltonian diffeomorphisms and Reeb flows},
Int. Math. Res. Not. IMRN 2010, 53--68.

\bibitem{HM}
F.~Hirzebruch, K.H.~Mayer,
\emph{{${\rm O}(n)$}-{M}annigfaltigkeiten, exotische Sph\"aren und Singularit\"aten},
Lecture Notes in Mathematics, No.~57, 1968.

\bibitem{R}
A.~Ritter,
\emph{Topological quantum field theory structure on symplectic homology},
arXiv:1003.1781.

\bibitem{SZ}
D.~Salamon, E.~Zehnder,
\emph{Morse theory for periodic solutions of Hamiltonian systems and the Maslov index},
Comm. Pure Appl. Math.~{\bf 45} (1992), no. 10, 1303â1360.

\bibitem{vK}
O.~van Koert,
\emph{Open books for contact five-manifolds and applications of contact homology},
Inaugural-Dissertation zur Erlangung des Doktorgrades der
Mathematisch-Naturwissenschaftlichen Fakult\"at der Universit\"at zu K\"oln (2005).
http://kups.ub.uni-koeln.de/1540/

\bibitem{vK2}
O.~van Koert,
\emph{Contact homology of Brieskorn manifolds},
Forum Math.~{\bf 20} (2008) 317--339.

\bibitem{V1}
C.~Viterbo,
\emph{Equivariant Morse Theory for Starshaped Hamiltonian Systems},
Trans.\,Amer.\,Math.\,Soc.~\textbf{311}, (1989), no.\,2, 621--655.

\bibitem{V}
C.~Viterbo,
\emph{Functors and computations in Floer homology with applications.\,I.}
GAFA \textbf{9} (1999), 985--1033.

\bibitem{W}
A.~Wasserman,
\emph{Equivariant differential topology.}
Topology \textbf{8} (1969), 127--150.

\end{thebibliography}
\end{document}